\documentclass[12pt]{article}
\usepackage{tikz}
\usepackage{amsmath}
\usepackage{amsopn}
\usepackage{amssymb}
\usepackage{latexsym}
\usepackage{amsfonts}
\usepackage{amsthm}
\usepackage{fullpage}
\usepackage{verbatim}
\usepackage{authblk}
\usepackage{multirow}
\usepackage{mathrsfs}
\usepackage[textsize=footnotesize]{todonotes}

\usepackage[colorlinks]{hyperref}
\usepackage[english]{babel}
\usepackage[utf8]{inputenc}
\usepackage{tikz}
\usepackage{mathtools}

\newcommand{\Dfn}[1]{\emph{\color{blue}#1}} % to highlight and collect definitions
\newcommand{\size}[1]{\left\lvert #1\right\rvert} % the size of a set
\DeclareMathOperator{\inumber}{\mathtt i}
\DeclareMathOperator{\dindex}{\mathtt t}
\DeclareMathOperator{\dualmaj}{\widehat{\mathtt maj}}
\DeclareMathOperator{\depth}{depth}
\DeclareMathOperator{\arcs}{\text Arcs}
\DeclareMathOperator{\Dstat}{\mathtt D}

\theoremstyle{definition}
\newtheorem{Theorem}{Theorem}

\newtheorem{Lemma}[equation]{Lemma}
\newtheorem{Proposition}{Proposition}

\theoremstyle{definition}
\newtheorem{Definition}[equation]{Definition}
\newtheorem{Definition-Remark}[equation]{Definition/Remark}
\newtheorem{Example}[equation]{Example}

\theoremstyle{remark}

\newtheorem{Remark}[equation]{Remark}

\numberwithin{equation}{section}
\numberwithin{figure}{section}

\newcommand{\F}{\mathbb{F}}

\newcommand{\PP}{\mathbb{P}}

\newcommand{\mc}[1]{\mathcal{#1}} % short for mathcal
\newcommand{\ms}[1]{\mathscr{#1}} % short for mathscript
\newcommand{\mb}[1]{\mathbb{#1}} % short for mathblackboard
\newcommand{\mt}[1]{\text{#1}}
\newcommand{\BM}[1]{\overline{\mathbb{B}}_{#1}}

\begin{document}

\title{A Geometric Interpretation of the Intertwining Number}

\author[1]{Mahir Bilen Can\thanks{Partially supported by a grant from Louisiana Board of Regents}}
\author[2]{Yonah Cherniavsky}
\author[3]{Martin Rubey\thanks{Supported by the Austrian Science Fund (FWF): P 29275}}

\affil[1]{\small Tulane University, New Orleans; mahirbilencan@gmail.com}
\affil[2]{\small Ariel University, Israel; yonahch@ariel.ac.il}
\affil[3]{\small Technische Universit\"{a}t Wien; martin.rubey@tuwien.ac.at}

%\normalsize

%\date{\today}

\maketitle

\begin{abstract}
We exhibit a connection between two statistics on set partitions,
the intertwining number and the depth-index. In particular, results link
the intertwining number to the algebraic geometry of Borel orbits.
Furthermore, by studying the generating polynomials of our statistics,
we determine the $q=-1$ specialization of a $q$-analogue of the Bell numbers.
Finally, by using Renner's $H$-polynomial of an algebraic monoid,
we introduce and study a $t$-analog of $q$-Stirling numbers.
\vspace{.2cm}

\end{abstract}

\section{Introduction}
%\subsection{Brief introduction and main result}

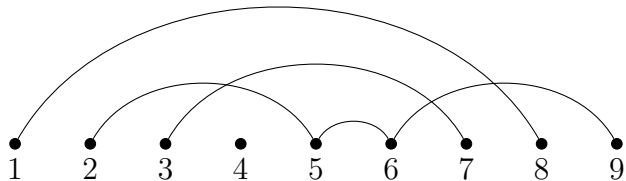
\begin{figure}
  \centering
% sage: A = SetPartition([[1,8],[2,5,6,9],[3,7],[4]])
% sage: A.set_latex_options(plot='linear', angle=30)
% sage: latex(A)
\begin{tikzpicture}[scale=1]
\node[below=.05cm] at (0,0) {$1$};
\node[draw,circle, inner sep=0pt, minimum width=4pt, fill=black] (0) at (0,0) {};
\node[below=.05cm] at (1,0) {$2$};
\node[draw,circle, inner sep=0pt, minimum width=4pt, fill=black] (1) at (1,0) {};
\node[below=.05cm] at (2,0) {$3$};
\node[draw,circle, inner sep=0pt, minimum width=4pt, fill=black] (2) at (2,0) {};
\node[below=.05cm] at (3,0) {$4$};
\node[draw,circle, inner sep=0pt, minimum width=4pt, fill=black] (3) at (3,0) {};
\node[below=.05cm] at (4,0) {$5$};
\node[draw,circle, inner sep=0pt, minimum width=4pt, fill=black] (4) at (4,0) {};
\node[below=.05cm] at (5,0) {$6$};
\node[draw,circle, inner sep=0pt, minimum width=4pt, fill=black] (5) at (5,0) {};
\node[below=.05cm] at (6,0) {$7$};
\node[draw,circle, inner sep=0pt, minimum width=4pt, fill=black] (6) at (6,0) {};
\node[below=.05cm] at (7,0) {$8$};
\node[draw,circle, inner sep=0pt, minimum width=4pt, fill=black] (7) at (7,0) {};
\node[below=.05cm] at (8,0) {$9$};
\node[draw,circle, inner sep=0pt, minimum width=4pt, fill=black] (8) at (8,0) {};
\draw[color=black] (6) to [out=120,in=60] (2);
\draw[color=black] (7) to [out=120,in=60] (0);
\draw[color=black] (4) to [out=120,in=60] (1);
\draw[color=black] (5) to [out=120,in=60] (4);
\draw[color=black] (8) to [out=120,in=60] (5);
\end{tikzpicture}
\caption{The arc-diagram of the set partition
  $A = 18 | 2569 | 37 | 4$.}
\label{F:introexample}
\end{figure}

This paper is concerned with the intertwining number of a set partition, which is a combinatorial statistic introduced by Ehrenborg and Readdy in~\cite{ER}. This statistic is among the combinatorial parameters on set partitions whose generating function is an important $q$-analog of the Stirling numbers of the second kind: 
\begin{align}\label{A:particular}
S_q(n,k)=
\begin{cases}
q^{k-1}S_q(n-1,k-1)+[k]_qS_q(n-1,k) & \ \text{ if $n,k\geq 1$};\\
\delta_{n,k} & \ \text{ if $n=0$ or $k=0$.}
\end{cases}
\end{align}
Here $\delta_{n,k}$ is the Kronecker's delta function.
As far as we know, this recurrence
has first appeared in a paper of Milne who showed that (\ref{A:particular})
has a combinatorial interpretation in terms of statistics on set partitions.
After Milne's work, many authors found interesting combinatorial statistics
whose (bi)generating polynomials satisfy the recurrence in (\ref{A:particular}),
see for example~\cite{WachsWhite}.
For more recent results on the $q$-Stirling numbers,
and an exposition of the history of Stirling numbers,
we recommend the articles~\cite{CR,CER1}.

In the present paper we connect the intertwining number to another statistic on the set partitions, 
namely, the depth-index, which was recently introduced and studied in~\cite{CC} by the first two authors. 
The depth-index, although defined in purely combinatorial terms, equals the dimension of the closure of a certain (doubled) Borel orbit, 
and thus, the intertwining number also receives a geometric interpretation. 
One purpose of this article is to show that the
depth-index is related in an
interesting way to other set partition
statistics.

Let us briefly set up the notation that is necessary to state our main
results.  Let $A$ be a set partition of $\{1,\dots, n\}$ into
\Dfn{blocks} $A_1,\dots,A_k$.  The minimal elements of the blocks are
the \Dfn{openers}, and the maximal elements are the \Dfn{closers} of
the set partition.  For example, $A = 18 | 2569 | 37 | 4$ in $\Pi_9$
has openers $1,2,3,4$ and closers $4,7,8,9$.

Let us assume that the elements of each block are listed in
increasing order, that is, two elements $i,j$ in a block are
consecutive if $j$ is the smallest element larger than $i$ in the
same block.  Then the \Dfn{arc diagram} of a set partition
$A\in\Pi_n$ is obtained by placing labels $1,\dots,n$ in this order
on a horizontal line, and connecting consecutive elements of each
block by arcs, as in Figure~\ref{F:introexample}.

The \Dfn{extended arc diagram} is obtained from the arc diagram by
adding a half-arc $(-\infty, i)$ from the far left to each opener
$i$, and a half-arc $(i, \infty)$ from each closer $i$ to the far
right.  These arcs are drawn in such a way that half-arcs to the left
do not cross, and half-arcs to the right do not cross either.  An
example is shown in Figure~\ref{fig:depth-index-intertwining-number}.

Two (generalized) arcs $(i,j)$ and $(k,\ell)$ \Dfn{cross} in $A$ if
$i < k < j < \ell$.  The total number of crossings in $A$ is the
\Dfn{intertwining
  number}\footnote{\url{http://www.findstat.org/St000490}} of $A$,
denoted $\inumber(A)$.

%Two crossing (generalized) arcs form a \Dfn{simple crossing} in $A$
%if $(i,j)$ and $(k,\ell)$ cross and there is no (generalized) arc
%$(r,s)$ in $A$ such that $i\leq r < k < j < s \leq\ell$.
%Essentially, a crossing of two arcs is simple, if it is not part of a
%triple of mutually crossing (generalized) arcs.

To indicate the geometric motive of our work, we will   
briefly mention the \Dfn{Bruhat-Chevalley-Renner order}
for set partitions: let $\mb{B}_n$ be the group of invertible upper
triangular $n\times n$ matrices and let $\ms{B}_n$ be the monoid of
upper triangular matrices with entries in $\{0,1\}$, having at most one
non-zero entry in each row and column.  Then, for $\sigma$ and $\tau$
in $\ms{B}_n$,
\[
  \sigma \leq \tau \iff%
  \mb{B}_n\sigma\mb{B}_n \subseteq \overline{ \mb{B}_n\tau\mb{B}_n},
\]
where $\overline{X}$ denotes the Zariski closure of $X$.  It was
shown by Renner~\cite[sec.8]{R} that this makes $\ms{B}_n$ into a
graded poset, the rank of $\sigma$ being
$\dim \mb{B}_n \sigma \mb{B}_n$.  Let $\ms{B}_n^{nil}$ be
the semigroup of nilpotent elements in $\ms{B}_n$.
Then the following simple bijection between
$\ms{B}^{nil}_n$ and $\Pi_n$ makes this order into an order of
$\Pi_n$: the matrix corresponding to the set partition $A$ has an
entry equal to $1$ in row $i$ and and column $j$ if and only if
$(i,j)$ is an arc of $A$.

%\todo{this is a duplication of the second paragraph on page 2}
A purely combinatorial way to describe the above partial order on set partitions was recently introduced in~\cite{CC}, 
where the rank function of the poset was given by a certain combinatorial statistic (on arc-diagrams), called the 
depth-index of $A$ and was denoted by $\dindex(A)$.
Now we are ready to outline the structure of this article and to mention our results which connect
the depth-index to the other statistics. 

First of all, in Section~\ref{S:two}, we relate the intertwining number to the depth-index; 
this is our Theorem~\ref{T:first main result}. One the one hand,  the depth-index $\dindex(A)$ 
is the rank function of the Bruhat-Chevalley-Renner order on the set partitions. 
On the other hand, the poset of doubled Borel orbits ordered by 
the containment relations on closures is ranked by the dimension function.
In other words, the depth-index $\dindex(A)$ gives the dimension of $\overline{ \mb{B}_n\tau\mb{B}_n}$ where $\tau$ is the upper-triangular partial permutation matrix which corresponds to the set partition $A$. Thus, it follows from Theorem~\ref{T:first main result} that the intertwining number $\inumber(A)$ is the rank function of the dual poset, 
$\inumber(A) = \text{codim}\, ( \overline{ \mb{B}_n\tau\mb{B}_n})$. Notice that from a computational point of view the intertwining number is simpler than the depth index.
In Section~\ref{S:three}, we give another combinatorial interpretation of the depth-index and of the intertwining number using so-called rank-control matrices. In Section~\ref{S:four}, we apply Theorem~\ref{T:first main result} to compute $q$-Bell numbers corresponding to the depth-index when $q=-1$. In Section~\ref{S:five}, we use Renner's $H$-polynomial of an algebraic monoid to introduce and study a new $t$-analog of $q$-Stirling numbers.
\vspace{.25cm}

\section{The intertwining number and the depth-index}\label{S:two}
From now on, we identify a set partition with its arc-diagram.
The blocks of a set partitions are chains in its arc-diagram.
We will frequently use the following well-known and very important,
albeit rather obvious fact:
If $A$ is an arc-diagram on $n$ vertices with $k$ arcs,
then the number of chains of
$A$ is $n-k$, while singletons are also considered as chains.
In this regard, we will denote by $\Pi_{n,k}$ ($k=1,\dots, n$)
the set of set partitions of $\{1,\dots, n\}$ with $k$ blocks,
or equivalently, the set of arc-diagrams on $n$ vertices
with $k$ chains.
%\todo{I think we don't actually use this, and it is really trivial}

In this section we exhibit the relationship between Ehrenborg and
Readdy's intertwining number, introduced in~\cite{ER}, and the depth
index, introduced in~\cite{CC}.

\begin{Definition}\label{D:intertwining and dindex}
Let $A$ be a set partition from $\Pi_{n,k}$.
Then the \Dfn{intertwining
number}\footnote{\url{http://www.findstat.org/St000490}}
$\inumber(A)$ of $A$ is the total number of crossings in the
extended arc diagram of $A$.
More formally, for a pair of disjoint sets $B$ and $C$ of integers,
the intertwining number is the cardinality of the set
  \[
    \left\{(b,c)\in B\times C\,:\,\{\min(b,c)+1,\dots,
      \max(b,c)-1\}\cap \left(B\cup C\right)=\emptyset\right\},
  \]
and the intertwining number of a set partition $A$ is the sum of
intertwining numbers of all pairs of blocks of $A$.

Let us denote by $\arcs(A)$ the set of arcs of $A$.
The \Dfn{depth index}\footnote{\url{http://www.findstat.org/St001094}}
$\dindex(A)$ of $A$ is
\begin{align}\label{A:dindex def}
    \sum_{i=1}^k (n-i)%
    -\sum_{v=1}^n \depth(v)%
    +\sum_{\alpha\in\arcs(A)} \depth(\alpha),
\end{align}
where $\depth(v)$, which is called the depth of a vertex $v$,
is the number of arcs $(i,j)\in\arcs(A)$ with $1\leq i<v<j\leq n$,
and $\depth(\alpha)$, which is called the depth of an arc $\alpha=(u,v)$,
is the number of arcs $(i,j)\in\arcs(A)$ with $1\leq i<u<v<j\leq n$.
\end{Definition}

The intertwining number is best understood by visualizing it
on the extended arc diagram of the set partition. 
An example is depicted in 
Figure~\ref{fig:depth-index-intertwining-number}.

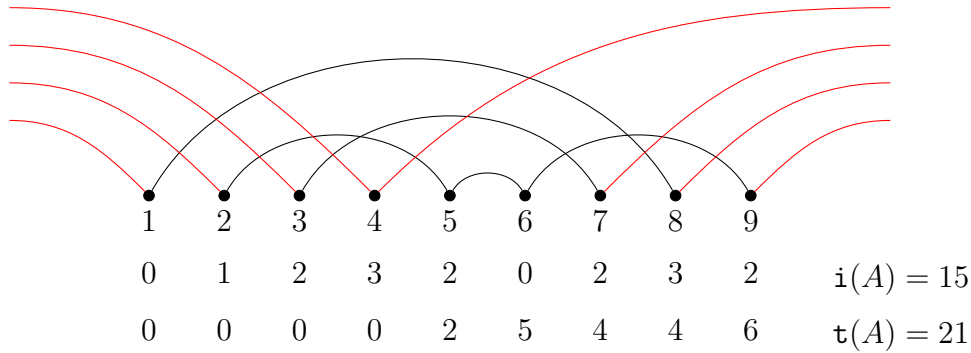
\begin{figure}[h]
  \centering
% sage: A = SetPartition([[1,8],[2,5,6,9],[3,7],[4]])
% sage: A.set_latex_options(plot='linear', angle=30)
% sage: latex(A)
\begin{tikzpicture}[scale=1]
\node[below=.05cm] at (0,0) {$1$};
\node[draw,circle, inner sep=0pt, minimum width=4pt, fill=black] (0) at (0,0) {};
\node[below=.05cm] at (1,0) {$2$};
\node[draw,circle, inner sep=0pt, minimum width=4pt, fill=black] (1) at (1,0) {};
\node[below=.05cm] at (2,0) {$3$};
\node[draw,circle, inner sep=0pt, minimum width=4pt, fill=black] (2) at (2,0) {};
\node[below=.05cm] at (3,0) {$4$};
\node[draw,circle, inner sep=0pt, minimum width=4pt, fill=black] (3) at (3,0) {};
\node[below=.05cm] at (4,0) {$5$};
\node[draw,circle, inner sep=0pt, minimum width=4pt, fill=black] (4) at (4,0) {};
\node[below=.05cm] at (5,0) {$6$};
\node[draw,circle, inner sep=0pt, minimum width=4pt, fill=black] (5) at (5,0) {};
\node[below=.05cm] at (6,0) {$7$};
\node[draw,circle, inner sep=0pt, minimum width=4pt, fill=black] (6) at (6,0) {};
\node[below=.05cm] at (7,0) {$8$};
\node[draw,circle, inner sep=0pt, minimum width=4pt, fill=black] (7) at (7,0) {};
\node[below=.05cm] at (8,0) {$9$};
\node[draw,circle, inner sep=0pt, minimum width=4pt, fill=black] (8) at (8,0) {};
\draw[color=black] (6) to [out=120,in=60] (2);
\draw[color=black] (7) to [out=120,in=60] (0);
\draw[color=black] (4) to [out=120,in=60] (1);
\draw[color=black] (5) to [out=120,in=60] (4);
\draw[color=black] (8) to [out=120,in=60] (5);
% left lines
\node[] (a) at (-2,1) {};
\node[] (b) at (-2,1.5) {};
\node[] (c) at (-2,2) {};
\node[] (d) at (-2,2.5) {};
\draw[color=red] (0) to [out=135,in=0] (a);
\draw[color=red] (1) to [out=135,in=0] (b);
\draw[color=red] (2) to [out=135,in=0] (c);
\draw[color=red] (3) to [out=135,in=0] (d);
% right lines
\node[] (e) at (10,1) {};
\node[] (f) at (10,1.5) {};
\node[] (g) at (10,2) {};
\node[] (h) at (10,2.5) {};
\draw[color=red] (8) to [out=45,in=180] (e);
\draw[color=red] (7) to [out=45,in=180] (f);
\draw[color=red] (6) to [out=45,in=180] (g);
\draw[color=red] (3) to [out=45,in=180] (h);
% crossing numbers
\node[below=0.75cm] at (0) {$0$};
\node[below=0.75cm] at (1) {$1$};
\node[below=0.75cm] at (2) {$2$};
\node[below=0.75cm] at (3) {$3$};
\node[below=0.75cm] at (4) {$2$};
\node[below=0.75cm] at (5) {$0$};
\node[below=0.75cm] at (6) {$2$};
\node[below=0.75cm] at (7) {$3$};
\node[below=0.75cm] at (8) {$2$};
\node[below=0.75cm] at (10,0) {$\inumber(A)=15$};
% depth index numbers
\node[below=1.5cm] at (0) {$0$};
\node[below=1.5cm] at (1) {$0$};
\node[below=1.5cm] at (2) {$0$};
\node[below=1.5cm] at (3) {$0$};
\node[below=1.5cm] at (4) {$2$};
\node[below=1.5cm] at (5) {$5$};
\node[below=1.5cm] at (6) {$4$};
\node[below=1.5cm] at (7) {$4$};
\node[below=1.5cm] at (8) {$6$};
\node[below=1.5cm] at (10,0) {$\dindex(A)=21$};
\end{tikzpicture}
\caption{The extended arc diagram of the set partition
  $A = 18 | 2569| 37|4$, together with the
  computation of the intertwining number and the depth index.}
  \label{fig:depth-index-intertwining-number}
\end{figure}

\begin{Remark}
The second sum appearing in (\ref{A:dindex def}), that is $\sum_{v=1}^n \depth(v)$,
coincides with the dimension exponent\footnote{\url{http://www.findstat.org/St000572}},
\[
  \sum_{B\text{ is a block of } A} (\max B - \min B + 1) - n.
\]
Moreover, the sum $\sum_{\alpha\in\arcs(A)} \depth(\alpha)$ is just
the number of
nestings\footnote{\url{http://www.findstat.org/St000233}} of $A$.
Note that the number of nestings is equidistributed with the number
of crossings\footnote{\url{http://www.findstat.org/St000232}}.  This
statistic and the dimension exponent both occur in the theory of
supercharacters.  More precisely, for any supercharacter indexed by a
set partition $A$, the dimension is given by the dimension exponent
of $A$, and the scalar product with a supercharacter indexed by the
same set partition equals the number of crossings of $A$, see~\cite{ACDS}.
\end{Remark}

Our main result is the following:
\begin{Theorem}\label{T:first main result}
  For any set partition $A\in\Pi_n$, we have
  \[
    \dindex(A) + \inumber(A) = \binom{n}{2}.
  \]
\end{Theorem}

For the proof it will be
convenient to refine the intertwining number and
the depth index as follows.
\begin{Definition}\label{D:Partials}
  Let $A$ be a set partition of $\{1,\dots,n\}$, and let
  $v\in\{1,\dots,n\}$.

  The \Dfn{partner} $u$ of an element is $0$ if $v$ is the minimal
  element of its block, otherwise it is the largest element in the
  same block smaller than $v$, that is, $(u,v)$ is an arc in the arc
  diagram.

  The \Dfn{partial intertwining number}, denoted by $\inumber_v(A)$,
  is the number of crossings of the arc (or half-arc) ending in $v$ with
  arcs or half-arcs whose smaller vertex $i$ is between $u$ and $v$.

  The \Dfn{partial depth index}, denoted by $\dindex_v(A)$,
  is the sum of the
  number of (proper) arcs $(i,j)$ with $u < i < j < v$ and the number
  $u$, which is the partner of $v$. 
  
%\todo{you added: \lq is a partner of $v$\rq\, which would make it wrong.  
%We are actually adding the number $u$.}
\end{Definition}

In Figure~\ref{fig:depth-index-intertwining-number}, the partial
intertwining numbers are written on the second line, below the
elements of the set partition.  Since every crossing is counted
precisely once, the sum of these numbers is the intertwining number
of $A$.  The partial depth indices are written below, on the third
line.

It is clear that the sum of partial intertwining numbers is the
intertwining number of the set partition.  The corresponding
statement for the depth index is also true:
\begin{Lemma}\label{lem:total-depth-indices}
  The sum of the partial depth indices of a set partition is equal to
  its depth index.
\end{Lemma}

Before we prove this lemma, let us note a second useful fact.

\begin{Lemma}\label{lem:sum-partial-intertwining-depth-index}
  For each $v\in\{1,\dots,n\}$, the sum of $\dindex_v(A)$ and
  $\inumber_v(A)$ equals $v-1$, the total number of vertices before
  $v$.
\end{Lemma}

Clearly, Theorem~\ref{T:first main result} follows at once from
Lemma~\ref{lem:total-depth-indices} and
Lemma~\ref{lem:sum-partial-intertwining-depth-index}.

\begin{proof}[Proof of~\ref{lem:sum-partial-intertwining-depth-index}]
  Let $u$ be the partner of $v$.  Then any arc $(i,j)$ with
  $u < i < v$ either satisfies $u < i < j < v$, and thus contributes
  $+1$ to the partial depth index, or it contributes precisely one
  crossing to $\inumber_v(A)$.
\end{proof}

\begin{proof}[Proof of~\ref{lem:total-depth-indices}]
  Let $A$ be a set partition of $\{1,\dots,n\}$ and let
  $v\in\{1,\dots,n\}$.
  Let $A'$ be the set partition obtained from $A$ by removing the
  last vertex with label $n$.

  We will now use induction, so, we proceed with the assumption that
  $\dindex(A')=\sum_{v=1}^{n-1}\dindex_v(A')$.  Moreover, by
  definition we have $\dindex_v(A') = \dindex_v(A)$ for $v<n$.
 (Clearly, if $n=1$, then there is nothing to prove.)

  If $n$ is a singleton block of $A$, then
  $\dindex(A)$ is obtained from $\dindex(A')$ by adding the
  number of arcs of $A$.
  Otherwise, we assume that $(m, n)$ is an arc in $A$, and there are
  $\nu$ arcs $(i,j)$ in $A$ with $m < i < j < n$.  Then
  \[
  \dindex(A) = \dindex(A') + (n-1) - (n-1-m) + \nu,
  \]
  since the new arc $(m,n)$ contributes $n-1$ to the first sum in
  the definition of the depth index, and the new arc increases the
  depth of each of the vertices between $m$ and $n$ by $1$.

  In both of these cases,
  $\dindex(A) = \dindex(A') + m + \nu = \dindex(A') +
  \dindex_n(A)$, hence the proof is finished.
\end{proof}

\section{The intertwining number and the rank control matrix}\label{S:three}

We start with setting up our notation.
\begin{Definition}
  For an $n\times n$ matrix $X$, let $X_{k,\ell}$
  denote the lower-left $k\times\ell$ submatrix of $X$.
  Then the \Dfn{rank control matrix} $R(X)=(r_{k,\ell})_{k,\ell=1}^n$
  is the $n\times n$ matrix with entries defined by
  $r_{k,\ell}:=rank\left(X_{k,\ell}\right)$.
\end{Definition}

As far as we know, the rank control matrix $R(X)$ 
was introduced by Melnikov in~\cite{M}.
A closely related
version is used in~\cite{P} and in~\cite{MillerSturmfels}
for describing the Bruhat-Chevalley order on symmetric groups.
Incitti~\cite{I} used it in his study of the Bruhat order on involutions.
After Incitti's work, the rank control matrix
is used in~\cite{BC},~\cite{C},~\cite{CT},~\cite{CCT}
for studying Bruhat orders on partial involutions and partial fixed-point-free involutions.
See~\cite{CR} for related work on the rook monoid.

Next, we introduce the ``inequalities statistic'' of an
arc-diagram.

\begin{Definition}
Let $A$ be a set partition from $\Pi_n$.
We denote by $M(A)=(m_{i,j})_{i,j=1}^n$
the $n\times n$ matrix defined by
$$
m_{i,j} =
\begin{cases}
1 & \text{ if $(i,j)$ is an arc in $A$;}\\
0 & \text{ otherwise.}
\end{cases}
$$
In this notation, the \Dfn{inequalities statistic} of $A$,
denoted by $\Dstat (A)$, is defined by
\[
\Dstat(A)=\left|\left\{(i,j)\,|\,2\leq i\leq n,\, 1\leq j\leq n-1\text{ and }
r_{i,j}\neq r_{i-1,j+1}\right\}\right|\,,
\]
Here, $r_{i,j}$'s are the entries of
rank control matrix $R(M(A)) = (r_{i,j})_{i,j=1}^n$.
\end{Definition}
In other words, $M(A)$ is the adjacency matrix
of $A$, regarded as the directed graph with
edges directed towards the vertices with bigger
labels.  The statistic $\Dstat(A)$ is
the total number of inequalities
along antidiagonals from south-west to north-east
of the rank control
matrix $R(M(A))$.

\begin{Proposition}\label{prop:explain-r}
  Let $A$ be a set partition of $\{1,\dots,n\}$ and
  let $(r_{i,j})_{i,j=1}^n$ denote the rank control
  matrix of $M(A)$.
  In this case, the entry $r_{i,j}$ ($i,j\in \{1,\dots, n\}$)
  is equal to the number of
  arcs $(k,l)$ in $A$ such that $i\leq k < l\leq j$.
\end{Proposition}

The proof of Proposition~\ref{prop:explain-r} follows from
the definition of rank control matrix of $M(A)$, so we omit
writing it. Nevertheless, we give an example that explains it.

\begin{Example}
For example, let $A$ denote the arc-diagram
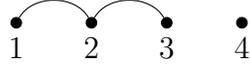
\begin{center}
%\begin{figure}[htp]
\begin{tikzpicture}[scale=1]
\node[below=.05cm] at (0,0) {$1$};
\node[draw,circle, inner sep=0pt, minimum width=4pt, fill=black] (0) at (0,0) {};
\node[below=.05cm] at (1,0) {$2$};
\node[draw,circle, inner sep=0pt, minimum width=4pt, fill=black] (1) at (1,0) {};
\node[below=.05cm] at (2,0) {$3$};
\node[draw,circle, inner sep=0pt, minimum width=4pt, fill=black] (2) at (2,0) {};
\node[below=.05cm] at (3,0) {$4$};
\node[draw,circle, inner sep=0pt, minimum width=4pt, fill=black] (3) at (3,0) {};
\draw[color=black] (0) to [out=60,in=120] (1);
\draw[color=black] (1) to [out=60,in=120] (2);
\end{tikzpicture}
%\end{figure}
\end{center}
Then $M(A)= \left[\begin{matrix}
0 &1 &0 &0\\
0 &0 &1 &0\\
0 &0 &0 &0\\
0 &0 &0 &0\end{matrix}\right]$ and $R(M(A))=\left[\begin{matrix}
0 &1 &2 &2\\
0 &0 &1 &1\\
0 &0 &0 &0\\
0 &0 &0 &0\end{matrix}\right]$.
The antidiagonals from south-west to north-east in he matrix $R(M(A))$ are:
$(0,1)$, $(0,0,2)$, $(0,0,1,2)$, $(0,0,1)$ and $(0,0)$.
Note that we skipped the initial and the last antidiagonal sequences
since they do not have any inequalities, hence they do not
contribute to our statistic.
Now, we have two inequalities in the third diagonal
$(0,0,1,2)$ ($0\neq 1$ and $1\neq 2$)
and one inequality in the first, second and fourth diagonals.
Thus, for this arc-diagram $A$ we have $\mathtt D(A)=5$.

Let us point out that the depth index of $A$ is equal to five as well:
$$
\mathtt t(A)=3+2-(0+0+0+0)+(0+0)=5\,.
$$
This is not a coincidence as we will show in our next result.
\end{Example}

\begin{Theorem}\label{T:rank control}
Let $A$ be an arc-diagram from $\Pi_n$.
By viewing $A$ as a directed graph
we let $R=(r_{i,j})_{i,j=1}^n$
denote the rank-control matrix of the
adjacency matrix of $A$.
In this case, the following relations hold true:
\begin{enumerate}
\item $\dindex (A) = \mathtt{D}(A)$, where
  \[
    \Dstat(A)=\left|\left\{(i,j)\,|\,2\leq i\leq n,\, 1\leq
        j\leq n-1\text{ and } r_{i,j}\neq
        r_{i-1,j+1}\right\}\right|\,;
  \]
\item $\inumber(A) = \mathtt{E}(A)$, where
\[
\mathtt E(A)=\left|\left\{(i,j)\,|\,2\leq i\leq j
+1\leq n\quad\textrm{and}\quad r_{i,j}
=r_{i-1,j+1}\right\}\right|\,,
\]
\end{enumerate}
\end{Theorem}
We divide the proof into two propositions.

\begin{Proposition}\label{Dt}
  Let $A$ be a set partition. Then
  \[
    \dindex(A)=\Dstat(A)\,.
  \]
\end{Proposition}
\begin{proof} For a fixed $j$ with $2\leq j\leq n$, we set
$\Dstat_j(A)=\left|\left\{i\,|\,2\leq i\leq n,\,\text{ and }r_{i,j}\neq r_{i+1,j-1}\right\}\right|$.
We will prove that $\dindex_j(A)=\Dstat_j(A)$ for every $j$, $2\leq j\leq n$,
where $\dindex_j(A)$ is the partial depth index defined as in Definition~\ref{D:Partials}.

If the vertex $j$ is not the endpoint of any arc, hence, the partner of $j$ is $0$,
or, equivalently, the $j$-th column of the matrix $M(A)$ consists of 0's only, then $t_j(A)$
is the number of arcs $(u,v)$ with $u<v<j$. In this case,
$\Dstat_j(A)$ equals the number of $1$'s to the left of the $j$-th column in the matrix $M(A)$.
Each such 1 in $M(A)$ corresponds to an arc in $A$, therefore, $\dindex_j(A)=\Dstat_j(A)$.

If the $j$-th column of $M(A)$ has a $1$ at the position $(w,j)$, then $(w, j)$ is an arc in $A$
for some $w$ ($w<j$) and furthermore there are $w$ inequalities of the form $r_{i,j}> r_{i+1,j-1}$
for each $i$ such that $1\leq i\leq w$.
Also, it is easily seen that each arc $(u,v)$, where $w<u<v<j$, contributes an
inequality of the form $r_{u,j}> r_{u+1,j-1}$.
Therefore, $\Dstat_j(A)=w+\left|\left\{(u,v)\,|\,w<u<v<j\,,\,m_{u,v}=1\right\}\right|$.
Since $w$ is the partner of the vertex $j$, we see that
$\dindex_j(A)=w+\left|\left\{(u,v)\,|\,w<u<v<j\,,\,(u,v)\,\,\text{is an arc in }\,A\right\}\right|$.
Therefore, in this case we have $\dindex_j(A)=\Dstat_j(A)$ also.

Now, since by Lemma~\ref{lem:total-depth-indices},
$\dindex(A)=\sum_{j=2}^n \dindex_j(A)$, and $\Dstat(A)=\sum_{j=2}^n \Dstat_j(A)$,
the proof is finished.
\end{proof}

By combining Theorem~\ref{T:first main result} and
Proposition~\ref{Dt} we
are able to express
the intertwining number in terms of the number of equalities in
anti-diagonals of the rank control matrix.
To this end, if $A$ is an arc-diagram from $\Pi_n$, then
let us denote by $\mathtt E(A)$ the following statistic:
$$
\mathtt E(A)=\left|\left\{(i,j)\,|\,2\leq i\leq j+1\leq
n\quad\textrm{and}\quad r_{i,j}=r_{i-1,j+1}\right\}\right|\,,
$$
where, as before, $r_{i,j}$'s are the entries of the rank control matrix
$R(M(A))$.

\begin{Proposition}\label{Et}
Let $A$ be an arc-diagram. Then
$$
\mathtt i(A)=\mathtt E(A)\,.
$$
\end{Proposition}
\begin{proof} Let $n$ be the number of vertices of $A$.
By Theorem~\ref{T:first main result},
we have $\mathtt i(A)={n\choose 2}-\mathtt t(A)$.
Using Theorem~\ref{Dt} we now have $\mathtt i(A)={n\choose 2}-\mathtt D(A)$.
To finish the proof we will show that
$\mathtt E(A)={n\choose 2}-\mathtt D(A)$.

  Now, notice that $i>j$ implies $r_{i,j} = r_{i-1,j+1}$
       since the adjacency matrix $M(A)$
       of $A$ is strictly upper triangular.
       (The edges are directed towards
       vertices with bigger indices.)
       Therefore, the pairs $(i,j)$ with
       $i > j$ do not contribute to $\Dstat(A)$.
       By definition, pairs $(i,j)$ with
       $i > j$ do not contribute to
       $\mathtt E(A)$ either.
       By arguing
       in a similar manner we see that
       a pair $(i,j)$ with $2 \leq i \leq j \leq n$
       contributes either to $\mathtt E(A)$
       or to $\mathtt D(A)$ but not to both. Clearly, the number of such pairs
       equals ${n \choose 2}$.
       This shows that $\mathtt E(A) + \mathtt D(A) = {n \choose 2}$,
       and the proof is finished.

\begin{comment}
Now, notice that if $i >j$, then
we have $r_{i,j} = r_{i-1,j+1}$
since the adjacency matrix $M(A)$
of $A$ is strictly upper triangular.
(The edges are directed towards
vertices with bigger indices.)
Therefore, the pairs $(i,j)$ with
$i > j$ do not contribute to $\Dstat(A)$.
At the same time, looking at the definition
of $\mathtt E(A)$, we see a pair $(i,j)$ with
$i > j$ does not contribute to
$\mathtt E(A)$ too. By arguing
in a similar manner we see that
a pair $(i,j)$ with $2 \leq i \leq j \leq n$
contributes either to $\mathtt E(A)$
or to $\mathtt D(A)$ but not to both of them
at the same time. Clearly, the number of such pairs
of integers is equal to ${n \choose 2}$.
This shows that $\mathtt E(A) + \mathtt D(A) = {n \choose 2}$,
and the proof is finished.
\end{comment}

\end{proof}

\begin{proof}[Proof of Theorem~\ref{T:rank control}]
This is a combination of Propositions~\ref{Dt} and~\ref{Et}.
\end{proof}

Before proceeding to the next section, we
briefly discuss the algebraic geometric
significance of the equalities
$\mathtt t(A)=\mathtt D(A)$ and
$\mathtt i(A)=\mathtt E(A)$.

\vspace{.5cm}

We already mentioned that the doubled
Borel group $\mb{B}_n\times \mb{B}_n$
acts on matrices via
\begin{align}\label{A:Action on matrices}
(B_1,B_2) \cdot X = B_1 X B_2^{-1} \ \text{ for } (B_1,B_2) \in \mb{B}_n\times \mb{B}_n,\
X\in \mt{Mat}_n,
\end{align}
and the action (\ref{A:Action on matrices}) restricts to
give an action on $\BM{n}$.
%We also mentioned that the nilpotent matrices in $\BM{n}$ is isomorphic to
%$\BM{n-1}$ as an algebraic monoid.

\begin{Proposition}\label{orbitinvariant}
Let $X$ and $Y$ be two matrices from $\mt{Mat}_n$
such that $Y=B XC$, where $B$ and $C$ are from
$\mb{B}_n$. If $X_{k\ell}$ and
$Y_{k\ell}$ denotes the lower-left $k\times\ell$ submatrices of $X$ and $Y$,
respectively, then
$$
rank\left(X_{k\ell}\right)=rank\left(Y_{k\ell}\right)\, \text{ for all } 1\leq k,\ell\leq n.
$$
 \end{Proposition}

\begin{proof}
Let us write the matrix $BXC$ in block form:
$$
\begin{pmatrix}
* &*\\
0_{k\times(n-k)} &B'
\end{pmatrix}\begin{pmatrix}
* &*\\
X_{k\ell} &*
\end{pmatrix}\begin{pmatrix}
C' &*\\
0_{(n-\ell)\times\ell} &*\end{pmatrix}=\begin{pmatrix}
* &*\\
B'X_{k\ell}C' &*
\end{pmatrix}\,,
$$
and therefore, $Y_{k\ell}=B'X_{k\ell}C'$. Matrices
$B' $ and $C'$ are invertible (upper-triangular $k\times k$
and $\ell\times\ell$ submatrices of $B$ and $C$ respectively), which implies that
$Y_{k\ell}$ and $X_{k\ell}$ have equal ranks.
\end{proof}

Proposition~\ref{orbitinvariant} shows that the rank control matrix $R$
is an invariant of a $\mathbb B_n \times\mathbb B_n$-orbit.
Let $\mc{N}_n$ denote the ${n\choose 2}$
dimensional affine space of
upper triangular $n\times n$ matrices which
are nilpotent.
If $Z$ is from $\mc{N}_n$, then the closure
$\overline{\mb{B}_n Z \mb{B}_n}$ in $\BM{n}$ is an affine algebraic
subvariety of $\mc{N}_n$.
As we mentioned in the introduction, such orbit closures
are parametrized by the arc-diagrams, and furthermore, the
Bruhat-Chevalley-Renner order
can be interpreted
in a combinatorial way on the arc-diagrams. All of this
is recorded in~\cite{CC}.
\begin{Remark}
There is another combinatorial method to determine the relation between two
elements in the Bruhat-Chevalley-Renner order.
Namely, it is sufficient to compare the corresponding rank control matrices componentwise.
More precisely,
for two $k\times\ell$ matrices $X=(x_{i,j})$, $Y=(y_{i,j})$
let us define
$X\leqslant Y$ if $x_{i,j}\leq y_{i,j}$ for all $i,j$,
$1\leq i\leq k$, $1\leq j\leq \ell$.
It can be shown (see~\cite{BC}) that
$\overline{\mathbb B_n Z_1\mathbb B_n}\subseteq \overline{\mathbb B_n Z_2\mathbb B_n}$
if and only if $R(Z_1)\leqslant R(Z_2)$.
Now, if $A$ and $B$ are two arc-diagrams,
then we write that $A\leqslant B$ if the corresponding
nilpotent partial permutation matrices $\sigma_A$ and $\sigma_B$
satisfy $\sigma_A \leqslant \sigma_B$.
In this notation, $A\leqslant B$
if and only if $R(M(A))\leqslant R(M(B))$.
\end{Remark}

Recall that the depth index statistic gives us the
dimensions of the $\mb{B}_n\times \mb{B}_n$-orbits,
therefore, by Theorem~\ref{T:first main result}, the
intertwining number of an arc-diagram gives the
codimension of the corresponding orbit closure in $\mc{N}_n$.
Note also that the dimension of the variety
$\overline{\mathbb B_n Z\mathbb B_n}$
equals $\dim \mc{N}_n$ minus the number of
algebraically independent polynomial
relations that defines the affine variety
$\overline{\mathbb B_n Z\mathbb B_n}$.
Therefore, by using Proposition~\ref{Et}, we obtain
the following result.

\begin{Proposition}\label{P:E}
Let $\overline{\mathbb B_n Z\mathbb B_n}$ ($Z\in \mc{N}_n$)
be a doubled Borel subgroup orbit. Then
the parameter $\mathtt E(R(Z))$ which
counts the number of equalities in the anti-diagonals in the
upper triangle of $R(Z)$ actually counts the number of
algebraically independent polynomial equations that
define the variety $\overline{\mathbb B_n Z\mathbb B_n}$.
\end{Proposition}

There is a proof of Proposition~\ref{P:E} that does not
use Theorem~\ref{T:first main result}. However, this
direct proof is somewhat long and it is very similar to the
proof of Theorem~7.6 of~\cite{BC}, so, we omit it.

%Let us just recall that the matrix $R$ provides the information
%about the ranks of lower-left submatrices of any matrix of a certain
%$\mathbb B_n\times\mathbb B_n$-orbit, and for any matrix $H$ the
%equality $rank(H)=m$ means that any $(m+1)\times (m+1)$ minor
%of $H$ is zero while there exists an $m\times m$ non-zero minor in $H$.

\section{$q=-1$ specialization}\label{S:four}
The next result of our article builds on the
connections between various set partition statistics
that we established so far. The $n$-th Bell number,
denoted by $B_n$ is the sum $B_n:=\sum_k S(n,k)$.
It follows from Theorem~\ref{T:first main result} that
the following polynomial is a $q$-analog of the Bell numbers:
\begin{align}\label{A:Bell}
B_n(q):= q^{n \choose 2} X_n\left( \frac{1}{q} \right),
\ \text{ where } X_n(q) = \sum_{\sigma \in \Pi_n} q^{\dindex(\sigma)}.
\end{align}
\begin{Theorem}\label{T:third main result}
Let $X_n(q)$ denote the rank generating series of
the depth-index statistic as in (\ref{A:Bell}). Then
$$
X_n(-1)=(-1)^{{n\choose 2}}B_n(-1)=\begin{cases}
1 & \text{if } n\equiv 0\,\,\text{or}\,\,1\,\, \text{or}\,\,3\,\,\text{or}\,\,10\,\,(\text{mod}\,12); \\
 -1 & \text{if } n\equiv 4\,\,\text{or}\,\,6\,\, \text{or}\,\,7\,\,\text{or}\,\,9\,\,(\text{mod}\,12); \\
 0 & \text{if } n\equiv 2\,(\text{mod}\,3).
 \end{cases}
 $$
\end{Theorem}
As we hinted at in the introduction,
the intertwining number on
$\Pi_{n,k}$ has the $q$-Stirling numbers
as its generating series,
\begin{align}\label{A:hinted}
S_q(n,k) = \sum_{A\in\Pi_{n,k}} q^{\inumber(A)} \qquad \text{($k=0,1\dots, n-1$).}
\end{align}
Another statistic with the same generating function is introduced by Milne
in~\cite{Mi}.
\begin{Definition}
Let $A=A_1|A_2 | \dots |A_k$ be a set partition from $\Pi_{n,k}$.
The \Dfn{dual major index}\footnote{\url{http://www.findstat.org/St000493}} of $A$
is defined by
$$
\dualmaj(A)=  \sum_{i=1}^k (i-1)\size{A_i}.
$$
\end{Definition}
The nomenclature is due to Sagan~\cite{Sa}.
Although it is not referred to as the `dual major index'
in~\cite{W}, $\dualmaj$ is the second of the three statistics
that are studied by Wagner.
In~\cite{St}, Steingr\'imsson refers to $\dualmaj$ as ``LOS",
the left opener statistic.

\vspace{.5cm}

A bijection $\phi:\Pi_{n,k}\to\Pi_{n,k}$ satisfying
\[
  \dualmaj(\phi(A)) = \inumber(A),
\]
and thus proving the equidistribution of the intertwining number and
the dual major index was found by Parviainen in~\cite{Parviainen}.
It is defined as follows. As before, let $\inumber_v(A)$ be the
number of crossings of the arc (or line) ending in $v$ with arcs or
lines whose smaller vertex $i$ is between the partner of $v$ and $v$.
Then the blocks of the set partition $\phi(A)$ are the sets of
elements sharing the same number $\inumber_v(A)$.  For example, the
image of the set partition $A= 18 | 2569 | 37|4$
from Figure~\ref{fig:depth-index-intertwining-number} is
$\phi(A)= 16|2 | 3579|48$ and the dual major
index of $\phi(A)$ is given by
$\dualmaj(\phi(A))=0\cdot 2+1\cdot 1 + 2\cdot 4 + 3\cdot 2 = 15$.

We are now ready to prove our third
result.

\begin{proof}[Proof of Theorem~\ref{T:third main result}]
Let $X_n(q)$ denote the generating function of the depth index
$\mathtt{t}$ and $Y_n(q)$ the generating function of the
intertwining number $\mathtt{i}$,
$$
X_n(q)=\sum_{A\in\Pi_n}q^{\mathtt{t}(A)}\quad,\quad Y_n(q)=\sum_{A\in\Pi_n}q^{\mathtt{i}(A)}\,.
$$
By Theorem~\ref{T:first main result}, we know that
\[
X_n(q)=q^{{n\choose 2}}Y_n\left(\frac{1}{q}\right)\,.
\]
Following Wagner, we denote the generating function of $\dualmaj$ by $B_n(q)$,
$$
B_n(q)=\sum_{A\in\Pi_n}q^{\dualmaj(A)}\,.
$$
Then
$$
Y_n(q)=B_n(q)\,.
$$
Wagner in~\cite{W} proved that
$$
B_n(-1)=\begin{cases}
(-1)^n & \text{if } n\equiv 0\,(\text{mod}\,3); \\
 (-1)^{n+1} & \text{if } n\equiv 1\,(\text{mod}\,3); \\
 0 & \text{if } n\equiv 2\,(\text{mod}\,3);
\end{cases}
$$
Notice that ${n\choose 2}=\frac{n(n-1)}{2}$ is even when $n\equiv 0$ or $1$ (mod $4$),
and it is odd when $n\equiv 2$ or $3$ (mod $4$).
Now, we have
$$
X_n(-1)=(-1)^{{n\choose 2}}B_n(-1)=\begin{cases}
1 & \text{if } n\equiv 0\,\,\text{or}\,\,1\,\, \text{or}\,\,3\,\,\text{or}\,\,10\,\,(\text{mod}\,12); \\
 -1 & \text{if } n\equiv 4\,\,\text{or}\,\,6\,\, \text{or}\,\,7\,\,\text{or}\,\,9\,\,(\text{mod}\,12); \\
 0 & \text{if } n\equiv 2\,(\text{mod}\,3);
 \end{cases}
 $$
 Obviously, $n\equiv 2$ (mod $3$) if and only if $n\equiv 2$ or 5 or 8 or 11 ( mod $12$).
This finishes the proof.
\end{proof}

\section{$H$-polynomials}\label{S:five}

In the previous section we considered 
the generating functions $B_n(q)$, $X_n(q)$,
and $Y_n(q)$ of the statistics  
$\dualmaj$, $\mathtt{t}$, and $\mathtt{i}$,
respectively. We noted that 
$B_n(q)=Y_n(q)= q^{{n\choose 2}} X_n(1/q)$. 
Here, 
we consider the following
generating polynomial of $q$-Stirling numbers:
$$
H_n(q,t):= \sum_{k=0}^{n} S_q(n,k) t^k.
$$
It follows from (\ref{A:hinted}) that $H_n(q,1) = B_n(q)$.
For $n=0,1,2$, the values of the polynomial $H_n(q,t)$ are given by 
\begin{align*}
H_0 (q,t) &= 1,\\
H_1(q,t) &= S_q(1,0) + S_q(1,1)t = t, \\
H_2(q,t) &= S_q(2,0) + S_q(2,1) t + S_q(2,2) t^2 = t + qt^2.
\end{align*}

In this section, by relating the depth index and the intertwining number to the numbers of rational points
of $\BM{n-1}$ over finite fields, we prove the following result.
\begin{Theorem}\label{T:four}
We have 
\begin{align*}
H_n \left(q, \frac{1}{1-q}\right) 
= \frac{ 1}{(1-q)^n},   
\end{align*}
for all $n\geq 0$.
\end{Theorem}

Let $M$ be an algebraic monoid with a reductive group of units denoted by $G$.
Let $B$ denote the Borel subgroup in $G$, $T$ denote the maximal torus contained
in $B$, and let $R$ denote the Renner monoid defined as $R:= \overline{N_G(T)} /T$,
where $N_G(T)$ is the normalizer of $T$ in $G$ and the bar indicates that we
are taking the closure in Zariski topology. 
The $H$-polynomial of the Renner monoid of $R$ is defined by
$$
H (R,q) := \sum_{\sigma \in R} q^{a(\sigma)} (q-1)^{b(\sigma)},
$$
where $a(\sigma)$ is the dimension of the unipotent part of the
orbit $B\sigma B$ and $b(\sigma)$ is the dimension of the diagonalizable
part of $B\sigma B$.

One can think of the $H$-polynomial of $R$ as a transformed 
``Hasse-Weil motivic zeta function'' of the projectivization $\PP( M-\{0 \})$ of $M$. 
Indeed, by treating $q$ as a power of a prime number,
for all sufficiently large $q$, $(H(R,q)-1)/(q-1)$ 
is equal to the number of rational points over $\F_q$ (the finite field
with $q$ elements) of $\PP(M-\{0\})$. See~\cite[Remark 3.2]{R2}.
For an application of this idea to the rook theory, we recommend~\cite{CR2}.

We will consider the simplest example, where $M$ is $\mt{Mat}_n$, the monoid
of all $n\times n$ matrices. Clearly, the projectivization of $\mt{Mat}_n$ is the $n^2-1$
dimensional projective space, hence its number of $\F_q$-rational points is given 
by $(q^{n^2}-1)/(q-1)$. The Renner monoid of $\mt{Mat}_n$ is the rook monoid 
$\ms{R}_n$. Furthermore, if $\sigma \in \ms{R}_n$, then $a(\sigma) = \ell(\sigma) - \mt{rank}(\sigma)$,
and $b(\sigma) = \mt{rank}(\sigma)$, where $\mt{rank}(\sigma)$ is the rank of $\sigma$ as a matrix.
Therefore, the $H$-polynomial of $\ms{R}_n$ becomes
\begin{align}\label{A:CR}
q^{n^2}=H(\ms{R}_n,q) = \sum_{\sigma \in \ms{R}_n} q^{ \ell(\sigma) - \mt{rank}(\sigma)} (q-1)^{\mt{rank}(\sigma)}.
\end{align}

An important consequence of the formula in (\ref{A:CR}) is that
if $X$ denotes any $\mb{B}_n\times \mb{B}_n$-stable smooth and irreducible subvariety of
$\mt{Mat}_n$, then $H(X,q)$ gives the number of $\F_q$-rational points of $X$.
In particular, we apply this observation to $\mc{N}_n$, the semigroup
of nilpotent elements in $\BM{n}$. Note that, there is a natural algebraic 
variety isomorphism between $\mc{N}_n$ and $\BM{n-1}$. The isomorphism
is given by the deleting of the first column and the last row of the elements of $\mc{N}_n$. 
Moreover, under this isomorphism, $\mb{B}_n\times \mb{B}_n$-orbits in $\mc{N}_n$ 
are isomorphically mapped onto $\mb{B}_{n-1}\times \mb{B}_{n-1}$-orbits in $\BM{n-1}$. 
Therefore, there is no loss of information to work with $\BM{n-1}$ instead of $\mc{N}_n$.
Now, on one hand we have the number of points of $\BM{n-1}$ over $\F_q$, 
\begin{align}\label{A1}
| \BM{n-1} | = q^{n\choose 2}.
\end{align}
On the other hand, we have that
\begin{align}\label{A2}
H(\BM{n-1},q) = \sum_{\sigma \in \ms{B}_{n-1}} q^{ \ell(\sigma) - \mt{rank}(\sigma)} (q-1)^{\mt{rank}(\sigma)}
= \sum_{\sigma \in \ms{B}_{n-1}} q^{ \mathtt {t}(\sigma) - \mt{rank}(\sigma)} (q-1)^{\mt{rank}(\sigma)}.
\end{align}
By combining (\ref{A1}) with (\ref{A2}) and noting that this is a polynomial
identity (since it holds true for all sufficiently large prime powers), we have 
\begin{align}\label{A:identity}
q^{n\choose 2}
= \sum_{\sigma \in \ms{B}_{n-1}} q^{ \mathtt {t}(\sigma) - \mt{rank}(\sigma)} (q-1)^{\mt{rank}(\sigma)}
 =\sum_{\sigma \in \ms{B}_{n-1}} q^{ \mathtt {t}(\sigma)} (1-1/q)^{\mt{rank}(\sigma)}.
\end{align}

Eqn.(\ref{A:identity}) suggests a
variation of the generating series of the statistic $\mathtt{t}$ and it provides 
a way to evaluate this generating series.

\begin{proof}[Proof of Theorem~\ref{T:four}]

Recall that 
\begin{align}\label{A:Stirling expanded}
q^{n\choose 2} \sum_{\sigma \in \ms{B}_{n-1},\ \text{rank} (\sigma) = k} q^{-\mathtt t (\sigma)}
= \sum_{A \in \Pi_{n,n-k}} q^{\mathtt i (A)} =S_q(n,n-k),
\end{align}
where $k\in \{1,\dots, n\}$. 
%(The reason for omitting $k=n$ in the summation is that there are no rank $n$ elements in $\ms{B}_{n-1}$.)

%Recall that $H_n(q,t)= \sum_{k=0}^{n-1} S_q(n,k)t^k$.
We multiply both sides of (\ref{A:Stirling expanded}) by $t^{n-k}$ and then sum 
over $k\in \{1,\dots, n\}$ to get
\begin{align*}
H_n(q,t)  &= 
q^{{n\choose 2}}  \sum_{k=1}^{n} 
\sum_{\sigma \in \ms{B}_{n-1},\ \text{rank} (\sigma) = k} q^{-\mathtt t (\sigma)} t^{n-k},
\end{align*}
which is equivalent to  
\begin{align}\label{A:get1}
H_n(q,t) &= 
q^{{n\choose 2}} t^n 
\sum_{\sigma \in \ms{B}_{n-1}} q^{-\mathtt t (\sigma)} t^{-\text{rank}(\sigma)}.
\end{align}
At the same time, since eqn. (\ref{A:identity}) is equivalent to 
$q^{-{n\choose 2}} =\sum_{\sigma \in \ms{B}_{n-1}} q^{- \mathtt {t}(\sigma)} (1-q)^{\mt{rank}(\sigma)}$,
by replacing $t$ with $(1-q)^{-1}$ in (\ref{A:get1}), we obtain
\begin{align*}
H_n \left(q, \frac{1}{1-q}\right) = 
q^{{n\choose 2}} \frac{q^{-{n\choose 2}}}{(1-q)^n}  
= 
\frac{ 1}{(1-q)^n}.
\end{align*}
This finishes the proof of our theorem for $n\geq 1$. 
The case of $n=0$ has already been computed before, hence, the proof is complete.
\end{proof}

\section{Final Remarks}

We close our paper by listing our new formulas for the intertwining number.
Let $A$ be an arc diagram on $n$ vertices with $k$ arcs.
\begin{enumerate}
\item $\mathtt i(A)={n\choose 2}-\mathtt t(A)$: This is the main result of the present paper.

\item $\mathtt i(A)={n\choose 2}-\ell(\sigma)$,
where $\sigma=(\sigma_1, \sigma_2,\dots, \sigma_n)$ is the partial permutation
corresponding to $A$. Recall that $\sigma$ is defined so that
$\sigma_j=i$ whenever $(i,j)$ is an arc in $A$.
$\ell(\sigma)$ is the length of the partial permutation $\sigma$.
See the formulas in the introduction.

\item $\mathtt i(A)={n\choose 2}-\dim \mb{B}_n \sigma \mb{B}_n$, where
$\sigma$ is as in the previous item.

\item $\mathtt i(A)=\mathtt E(A)$, where $\mathtt E(A)$ is the number of equalities
in the antidiagonals of the upper-triangle of the rank control matrix of $A$.
See Section~\ref{S:three}.

\item $\mathtt i(A)={n\choose 2}-\mathtt c(A)$.
The statistic $\mathtt c(A)$ is defined in~\cite{CC}. We
recall its definition here for completeness.
Let $\alpha$ be an arc in $A$ and let $cross(\alpha)$ denote
the total number of chains that are crossed by $\alpha$ in $A$.
Note that a chain can be crossed at most twice. In this case,
we consider it as a single crossing.
Let $A$ be an arc-diagram on $n$ vertices with $k$ arcs denoted
by $\alpha_1,\alpha_2,\dots,\alpha_k$ and $n-k$ chains denoted
$\beta_1,\beta_2,\dots,\beta_{n-k}$.  The crossing-index
of $A$ is defined by the formula
$$
\mathtt{c}(A)=\sum_{i=1}^k (n-i)-\sum_{j=1}^{n-k} depth(\beta_j)-\sum_{m=1}^k cross(\alpha_m),
$$
where $depth(\beta_j)$ ($j=1,\dots, n-k$) is the number of arcs
that are above $\beta_j$.
It is easy to show (by induction) that $\mathtt{c}(A)=\mathtt t(A)$.
For details, see~\cite{CC}.

\end{enumerate}

\end{document}